\documentclass[12 pt]{article}
\usepackage{mathrsfs}
\usepackage{amsfonts}
\usepackage{amssymb}
\usepackage{amsmath}
\usepackage{amsthm}
\usepackage{fullpage}
\usepackage{graphicx}
\usepackage{wrapfig}
\usepackage{caption}

\newtheorem{theorem}{Theorem}[section]
\newtheorem{lemma}{Lemma}[section]

\newtheorem{corollary}{Corollary}[section]

\newtheorem{remark}{Remark}[section]

\title{Refined Solutions of Time Inhomogeneous Optimal Stopping Games via Dirichlet Form}
\author{Yipeng Yang\thanks{Department of Mathematics, University of Missouri-Columbia, Columbia, Missouri, 65211 (yangyip@missouri.edu)}}

\begin{document}
\maketitle
\begin{abstract}
The properties of value functions of time inhomogeneous optimal
stopping problem and zero-sum game (Dynkin game) are studied through
time dependent Dirichlet form. Under the absolute continuity
condition on the transition function of the underlying diffusion
process and some other assumptions, the refined solutions without
exceptional starting points are proved to exist, and the value
functions of the optimal stopping and zero-sum game, which are
finely and cofinely continuous, are characterized as the solutions
of some variational inequalities, respectively.
\end{abstract}


\section{Introduction}

Let ${\bf M}=(X_t,P_{(s,x)})$ be a diffusion process on a locally
compact separable metric space $\mathbb{X}$. For two finely
continuous functions $g,h$ on $[0,\infty)\times\mathbb{X}$ and a
constant $\alpha>0$, define the following return functions of
optimal stopping games:
\begin{eqnarray}
J_{(s,x)}(\sigma)&=&E_{(s,x)}(e^{-\alpha\sigma}g(s+\sigma,X_{s+\sigma}))\label{voptstp}\\
J_{(s,x)}(\tau,\sigma)&=&E_{(s,x)}\left[e^{-\alpha(\tau\wedge\sigma)}\left(g(s+\sigma,X_{s+\sigma})I_{\tau>\sigma}+h(s+\tau,X_{s+\tau})I_{\tau\leq\sigma}\right)\right].\label{voptgame}
\end{eqnarray}
The values of the stopping games are defined as
$\tilde{e}_g=\sup_\sigma J_{(s,x)}(\sigma)$ and
$\tilde{\bar{w}}=\sup_\sigma\inf_\tau J_{(s,x)}(\tau,\sigma)$,
respectively. This kind of optimal stopping problems have been
continually developed due to its broad application in finance,
resource control or production management.

In the time homogeneous case, where ${\bf M},g,h$ are all time
homogenous, it is well known that $\tilde{e}_g$ is a quasi
continuous version of the solution of a variational inequality
problem formulated in terms of the Dirichlet form, see Nagai
\cite{Nagai78}. This result was successfully extended by Zabczyk
\cite{Zab84} to Dynkin game (zero sum game) where $\tilde{\bar{w}}$
was shown to be the quasi continuous version of the solution of a
certain variational inequality problem. In their work, there always
exist an exceptional set $N$ of starting points of ${\bf M}$. In
2006, Fukushima and Menda \cite{Fuku06} showed that if the
transition function of ${\bf M}$ is absolutely continuous with
respect to the underlying measure ${\bf m}$, then there does not
exist the exceptional set $N$, and $\tilde{e}_g$ and
$\tilde{\bar{w}}$ are finely continuous with any starting point of
${\bf M}$.

However, more work is needed to extend these results to the time
inhomogeneous case, especially the characteristics of the value
function. Using the time dependent Dirichlet form (generalized
Dirichlet form), Oshima \cite{Oshima06} showed that under some
conditions, $\tilde{e}_g$ (also $\tilde{\bar{w}}$) is still finely
and cofinely continuous with quasi every starting point of ${\bf
M}$, and except on an exceptional set $N$, $\tilde{e}_g$ (also
$\tilde{\bar{w}}$) is characterized as a version of the solution of
a variational inequality problem.

Recently, Palczewski and Stettner
\cite{Palczewski10}\cite{Palczewski11} used the penalty method to
characterize the continuity of the value function of a time
inhomogeneous optimal stopping problem. In their work, the
underlying process ${\bf M}$ is assumed to satisfy the Feller
continuity property. Lamberton \cite{Lamberton09} derived the
continuity property of the value function of a one-dimensional
optimal stopping problem, and the value function was characterized
as the unique solution of a variational inequality in the sense of
distributions. However, that result was difficulty to be extended to
multi-dimensional diffusions.

In all the afore mentioned work, the property of the time
inhomogeneous value functions along the dimension $t$ (time) were
not further studied. In this paper, through the time dependent
Dirichlet form, it is showed that under the absolute continuity
condition on the transition probability function $p_t$ and some
other assumptions, the value functions do belong to the functional
space $\mathscr{W}$, see (\ref{spaceW}). Further it is showed that
Oshima's \cite{Oshima06} results still hold and there does not exist
the exceptional set for the starting points of ${\bf M}$. This
result is then applied in Section \ref{exam} to the
 time inhomogeneous optimal stopping games where
the underlying process is a multi-dimensional time inhomogeneous Ito
diffusion.

\section{Time Dependent Dirichlet Form}\label{tidirichlet}

In this section we define the settings for the time dependent
Dirichlet form that are similar to those in \cite{Oshima06},
although some results from \cite{Stannat99}, whose notions are
different, will be used later. Let $\mathbb{X}$ be a locally compact
separable metric space and ${\bf m}$ be a positive Radon measure on
$\mathbb{X}$ with full support. For each $t\geq 0$, define
$(E^{(t)},F)$ as an ${\bf m}$-symmetric Dirichlet form on
$H=L^2(\mathbb{X};{\bf m})$ and for any $u\in F$, we assume that
$E^{(t)}(u,u)$ is a measurable function of $t$ and satisfies
\begin{displaymath}
\lambda^{-1}\|u\|_F^2\leq E_1^{(t)}(u,u)\leq\lambda\|u\|_F^2
\end{displaymath} for some constant $\lambda>0$, where ${E}_\alpha^{(t)}(u,v)={E}^{(t)}(u,v)+\alpha(u,v)_{\bf m},\ \alpha>0$,
and the $F$ norm is defined to be $\|u\|_{F}^2={E}_1^{(0)}(u,u)$. We
also assume that $F$ is regular and local in the usual sense
\cite{Fuku11}.

Define $F'$ as the dual space of $F$, then it can be seen that
$F\subset H=H'\subset F'$. For each $t$, there exists an operator
$L^{(t)}$ from $F$ to $F'$ such that
\begin{displaymath}
-( {L}^{(t)}u,v)={E}^{(t)}(u,v), \quad u,v\in F.
\end{displaymath}  Further the $F'$ norm is defined as
\begin{displaymath}
\|v\|_{F'}=\sup_{\|u\|_{F}=1}\{(v,u)\},
\end{displaymath} where $(v,u)$ denotes the canonical coupling
between $v\in F'$ and $u\in F$.

Define the spaces
\begin{displaymath}\mathscr{H}=\{\varphi(t,\cdot)\in
H: \|\varphi\|_{\mathscr{H}}<\infty\},
\end{displaymath} where
\begin{displaymath}
\|\varphi\|_{\mathscr{H}}^2=\int_{\mathbb{R}}\|\varphi(t,\cdot)\|_{H}^2dt,
\end{displaymath}
and
\begin{displaymath}
\mathscr{F}=\{\varphi(t,x)\in F:\|\varphi\|_{\mathscr{F}}<\infty \},
\end{displaymath} where
\begin{displaymath}
\|\varphi\|_{\mathscr{F}}^2=\int_{\mathbb{R}}\|\varphi(t,\cdot)\|_{F}^2dt.
\end{displaymath}
Clearly $\mathscr{F}\subset
\mathscr{H}=\mathscr{H}'\subset\mathscr{F}'$ densely and
continuously, where $\mathscr{H}',\mathscr{F}'$ are the dual spaces
of $\mathscr{H},\mathscr{F}$ respectively.

For any $\varphi\in\mathscr{F}$, considering $\varphi$ as a function
of $t\in\mathbb{R}$ with values in $F$, the distribution derivative
$\partial \varphi/\partial t$ is considered as a function of
$t\in\mathbb{R}$ with values in $F'$ such that
\begin{displaymath}
\int_{\mathbb{R}}\frac{\partial\varphi}{\partial
t}(t,\cdot)\xi(t)dt=-\int_{\mathbb{R}}\varphi(t,\cdot)\xi'(t)dt,
\end{displaymath} for any $\xi\in C_0^\infty(\mathbb{R})$. Then we
can define the space $\mathscr{W}$ as
\begin{equation}\label{spaceW}
\mathscr{W}=\{\varphi(t,x)\in\mathscr{F}: \frac{\partial
\varphi}{\partial t}\in\mathscr{F}',
\|\varphi\|_{\mathscr{W}}<\infty \},
\end{equation} where
\begin{displaymath}
\|\varphi\|_{\mathscr{W}}^2=\left\|\frac{\partial\varphi}{\partial
t}\right\|_{\mathscr{F}'}^2+\|\varphi\|_{\mathscr{F}}^2.
\end{displaymath} Since $\mathscr{F}$ and
$\mathscr{F}'$ are Banach spaces, it is easy to see that
$\mathscr{W}$ is also a Banach space. Further, $\mathscr{W}$ is
dense in $\mathscr{F}$.

We further define the bilinear form $\mathcal{E}$ by
\begin{equation}
\mathcal{E}(\varphi,\psi)=\left\{\begin{aligned}-&\langle
\frac{\partial \varphi}{\partial
t},\psi\rangle+\int_{\mathbb{R}}{E}^{(t)}(\varphi(t,\cdot),\psi(t,\cdot))dt,\quad
\varphi\in\mathscr{W},\psi\in\mathscr{F},\\
&\langle \frac{\partial \psi}{\partial
t},\varphi\rangle+\int_{\mathbb{R}}{E}^{(t)}(\varphi(t,\cdot),\psi(t,\cdot))dt,\quad
\varphi\in\mathscr{F},\psi\in\mathscr{W},
\end{aligned}\right.
\end{equation} where $\langle \frac{\partial \varphi}{\partial t},\psi\rangle=\int_{\mathbb{R}}\left(\frac{\partial
\varphi}{\partial t},\psi\right)dt$. We call
$(\mathcal{E},\mathscr{F})$ a time dependent Dirichlet form on
$\mathscr{H}$ \cite{Oshima06}.

As in \cite{Oshima06} we may introduce the time space process ${
Z}_t=(\tau(t),{ X}_t)$ on the domain
$\mathbb{Z}=\mathbb{R}\times\mathbb{X}$ with uniform motion
$\tau(t)$, then the resolvent $R_\alpha f$ of ${ Z}_t$ defined by
\begin{equation}\label{resolR}
R_\alpha f(s,x)=E_{(s,x)}\left(\int_0^\infty e^{-\alpha t}f(s+t,{
X}_{s+t})dt\right),\ \ (s,x)=z,\ \ f\in\mathscr{H},
\end{equation}
satisfies
\begin{equation}\label{res_id}
\left(\alpha-\frac{\partial}{\partial t}-{L}^{(t)}\right)R_\alpha
f(t,{ x})=f(t,{ x}),\ \forall t\geq 0.
\end{equation}
Furthermore, $R_\alpha f$ is considered as a version of $G_\alpha
f\in\mathscr{W}$, where $G_\alpha$ is the resolvent associated with
the form $\mathcal{E}_\alpha(,)=\mathcal{E}(,)+\alpha(,)_\nu$ and it
satisfies
\begin{equation}\label{D_res_eq}
\mathcal{E}_\alpha(G_\alpha f,\varphi)=(f,\varphi)_\nu,\quad \forall
\varphi\in\mathscr{F},
\end{equation} where $d\nu(t,{ x})=dtd{\bf m}({ x})$. We may
write $(\cdot,\cdot)_\nu$ as $(\cdot,\cdot)_{\mathscr{H}}$ to
indicate it as the inner product in $\mathscr{H}$.

%By the definition of $\mathcal{E}_\alpha$ and ${E}_\alpha^{(t)}$,
%Eq.(\ref{D_res_eq}) is equivalent to
%\begin{displaymath}
%-\left(\frac{\partial}{\partial t}G_\alpha
%f(t,\cdot),u\right)+{E}_\alpha^{(t)}(G_\alpha
%f(t,\cdot),u)=(f(t,\cdot),u),\quad \forall t\geq 0,\ \forall
%u\in F.
%\end{displaymath} The dual resolvent $\hat{G}_\alpha
%f\in\mathscr{W}$ can be similarly defined as a solution of
%\begin{displaymath}
%\left(\frac{\partial}{\partial t}\hat{G}_\alpha
%f(t,\cdot),u\right)+{E}_\alpha^{(t)}(u,\hat{G}_\alpha
%f(t,\cdot))=(f(t,\cdot),u),\quad \forall t\geq 0,\ \forall u\in
%F.
%\end{displaymath}
%For $f\in\mathscr{F}$ (respectively $f\in\mathscr{H}$) we have
%$\lim_{\alpha\to\infty}\alpha G_\alpha f=f$ in $\mathscr{F}$
%(respectively in $\mathscr{H}$). Similar results also hold for the
%dual resolvent $\hat{G}_\alpha$.

%By Theorem 1.3 in \cite{Oshima06}, there exist diffusion processes
%${ M}=({ Z}_t,P_{ z})$ and $\hat{ M}=(\hat{ Z}_t,P_{ z})$ on
%$\mathbb{Z}$ such that the resolvents $R_\alpha f$ and
%$\hat{R}_\alpha f$ of ${ M}$ and $\hat{ M}$ are quasi continuous
%modifications of $G_\alpha f$ and $\hat{G}_\alpha f$, respectively.
%Furthermore, if we let ${ Z}_t=(\tau(t),{ X}_t)$ and $\hat{
%Z}_t=(\hat{\tau}(t),\hat{ X}_t)$ be the decompositions of ${ Z}_t$
%and $\hat{ Z}_t$, then $\tau(t)=\tau(0)+t$ and
%$\hat{\tau}(t)=\hat{\tau}(0)-t$.

We now define $\mathcal{A}$ as a bilinear form on
$\mathscr{F}\times\mathscr{F}$ by
\begin{displaymath}
\mathcal{A}(\varphi,\psi)=\int_{\mathbb{R}}{E}^{(t)}(\varphi(t,\cdot),\psi(t,\cdot))dt,
\end{displaymath}
and set
$\mathcal{A}_\alpha(\varphi,\psi)=\mathcal{A}(\varphi,\psi)+\alpha(\varphi,\psi)_{\mathscr{H}}$,
then it can be seen that
$\mathcal{E}_\alpha(\varphi,\psi)=-\langle\frac{\partial\varphi}{\partial
t},\psi\rangle+\mathcal{A}_\alpha(\varphi,\psi)$ if
$\varphi\in\mathscr{W},\psi\in\mathscr{F}$, and
$\mathcal{E}_\alpha(\varphi,\psi)=\langle\frac{\partial\psi}{\partial
t},\varphi\rangle+\mathcal{A}_\alpha(\varphi,\psi)$ if
$\varphi\in\mathscr{F},\psi\in\mathscr{W}$. Also notice that if
$\varphi\in\mathscr{W}$, $\langle\frac{\partial\varphi}{\partial
t},\varphi\rangle=0$, hence
$\mathcal{E}_\alpha(\varphi,\varphi)=\mathcal{A}_\alpha(\varphi,\varphi)$
in this case, see Corollary 1.1 in \cite{Oshima04}.

A function $\varphi\in\mathscr{F}$ is called $\alpha$-potential if
$\mathcal{E}_\alpha(\varphi,\psi)\geq 0$ for any nonnegative
function $\psi\in\mathscr{W}$. Denote by $\mathscr{P}_\alpha$ the
family of all $\alpha$-potential functions. A function
$\varphi\in\mathscr{F}$ is called $\alpha$-excessive if and only if
$\varphi\geq 0$ and $nG_{n+\alpha}\varphi\leq\varphi$ a.e. for any
$n\geq 0$. For any $\alpha$-potential $\varphi\in\mathscr{F}$,
define its $\alpha$-excessive modification as
\begin{displaymath}\tilde{\varphi}=\lim_{n\to \infty}nR_{n+\alpha} \varphi.\end{displaymath}

For any function $g\in\mathscr{H}$, let
\begin{displaymath}
\mathscr{L}_g=\{\varphi\in\mathscr{F}: \varphi\geq g\ \ \nu \text{
a.e.}\},
\end{displaymath} then the following result holds (see Lemma 1.1 in
\cite{Oshima06}):
\begin{lemma}\label{lemmagep}
For any $\epsilon>0$ and $\alpha>0$, there exists a unique function
$g_\epsilon^\alpha\in\mathscr{W}$ such that
\begin{equation}\label{gepsilon}
-\left(\frac{\partial g_\epsilon^\alpha}{\partial
t},u\right)+E_\alpha^{(t)}(g_\epsilon^\alpha(t,\cdot),u)
=\frac{1}{\epsilon}\left((g_\epsilon^\alpha(t,\cdot)-g(t,\cdot))^-,u\right)
\end{equation} for any $u\in F$.
\end{lemma}
As a consequence, $g_\epsilon^\alpha$ solves
\begin{equation}\label{gHnorm}\mathcal{E}_\alpha(g_\epsilon^\alpha,\psi)=\frac{1}{\epsilon}((g_\epsilon^\alpha-g)^-,\psi)_{\mathscr{H}}, \ \ \forall
\psi\in\mathscr{F},
\end{equation} see Proposition 1.6 in \cite{Stannat99}.

%For each $\epsilon>0$ define the set  $B_\epsilon =
%\{(t,x):g_\epsilon^\alpha(t,x)<g(t,x)\}$, and we assume that
%\begin{equation}\label{assumset}\sup_{\epsilon>0}\frac{\nu(B_\epsilon)}{\epsilon}<\infty,\end{equation} where
%$g_\epsilon^\alpha$ solves (\ref{gepsilon}).

By Theorem 1.2 in \cite{Oshima06}, $e_g=\lim_{\epsilon\to
0}g_\epsilon^\alpha$ converges increasingly, strongly in
$\mathscr{H}$ and weakly in $\mathscr{F}$, and furthermore, $e_g$ is
the minimal function of $\mathscr{P}_\alpha\cap\mathscr{L}_g$
satisfying
\begin{equation}\label{viqh0}
\mathcal{A}_\alpha (e_g,e_g)\leq\mathcal{E}_\alpha(e_g,\psi),\quad
\forall \psi\in\mathscr{L}_g\cap\mathscr{W}.
\end{equation}

Given any open set $A\in\mathbb{Z}$, the capacity of $A$ is defined
by
\begin{displaymath} Cap(A)=\mathcal{E}_\alpha(e_{I_A},\psi),\
\psi\in\mathscr{W}, \psi=1 \text{ a.e. on }A.
\end{displaymath}

If $\varphi\in\mathscr{F}$ is an $\alpha$-potential, then there
exists a positive Radon measure $\mu_\varphi^\alpha$ on $\mathbb{Z}$
such that
\begin{displaymath}
\mathcal{E}_\alpha(\varphi,\psi)=\int_{\mathbb{Z}}\psi(z)d\mu_\varphi^\alpha(z)\quad\text{for
any }\psi\in C_0(\mathbb{Z})\cap\mathscr{W}
\end{displaymath} By Lemma 1.4 in
\cite{Oshima06}, $\mu_{\varphi}^\alpha$ does not charge any Borel
set of zero capacity. Put $e_A=e_{I_A}$ and
$\mu_A^\alpha=\mu_{e_A}^\alpha$, then the capacity of the set $A$
can also be defined by
\begin{displaymath}
Cap(A)=\mu_{A}^\alpha(\bar{A}).\end{displaymath} The notion of the
capacity is extended to any Borel set by the usual manner. A set is
called exceptional if it is of zero capacity. If a statement holds
everywhere except on an exceptional set, we say the statement holds
quasi-everywhere (q.e.).

\section{The Time Inhomogeneous Stopping Games}
In this section we will characterize the properties of the value
functions $\tilde{e}_g=\sup_\sigma J_{(s,x)}(\sigma)$ and
$\tilde{\bar{w}}=\sup_\sigma\inf_\tau J_{(s,x)}(\tau,\sigma)$ of the
time inhomogeneous stopping games. We first assume that the
transition probability function $p_t$ of the process $X_t$ satisfies
the \emph{absolute continuity} condition:
\begin{equation}\label{abscont}
p_t(x,\cdot)\ll m(\cdot),\quad \forall t.
\end{equation} In fact,
the Feller property in \cite{Palczewski10} implies the absolute
continuity condition on $p_t$, see, e.g., page 165 of \cite{Fuku11}.

\subsection{The Time Inhomogeneous Optimal Stopping Problem}
Consider $\tilde{e}_g(z)=\sup_\sigma J_{z}(\sigma)$ where
\begin{equation}
J_{z}(\sigma)=J_{(s,x)}(\sigma)=E_{(s,x)}(e^{-\alpha\sigma}g(s+\sigma,X_{s+\sigma})).
\end{equation}
Oshima showed that (see Theorem 3.1 in \cite{Oshima06}) if
$g\in\mathscr{F}$ is quasi continuous and $\mathscr{L}_g\cap
\mathscr{W}\neq\phi$, then $\tilde{e}_g(z)\in\mathscr{F}$
 is finely and cofinely continuous q.e., and $e_g$ solves the variational inequality (\ref{viqh0}).
In what follows we give conditions under which
$\tilde{e}_g\in\mathscr{W}$ and Oshima's result holds without the
exceptional set.

It is assumed that $g\in\mathscr{W}$ is a finely continuous function
on ${\mathbb{Z}}$ such that
\begin{equation}\label{gbdd}
|g(t,x)|\leq \varphi(t,x),
\end{equation} for some finite $\alpha$-excessive function $\varphi\in\mathscr{W}$
on $\mathbb{Z}$. We also assume that there exists a constant $K$
such that
\begin{equation}\label{assumHnorm}
\sup_{\epsilon>0}\frac{1}{\epsilon}\|(g_\epsilon^\alpha-g)^-\|_{\mathscr{H}}\leq
K\|g\|_{\mathscr{H}},
\end{equation}  where
$g_\epsilon^\alpha$ solves (\ref{gepsilon}). In the rest of this
section, the notion $K_i$ for some index $i$ denotes a constant.

\begin{lemma}\label{egW}
Under the assumptions (\ref{gbdd}) and (\ref{assumHnorm}),
$e_g\in\mathscr{W}$.
\end{lemma}
\begin{proof}
It has been proved that $e_g\in \mathcal{L}_g\cap P_\alpha$, and
$e_g\in\mathscr{F}$, see Theorem 1.2 of \cite{Oshima06} or
Proposition 1.7 of \cite{Stannat99}. Furthermore,
\begin{displaymath}
\sup_\epsilon\|g_\epsilon^\alpha-\varphi\|_{\mathscr{F}}\leq
K_1\|\varphi\|_{\mathscr{W}},
\end{displaymath} and
\begin{displaymath}
\sup_\epsilon \|g_\epsilon^\alpha\|_{\mathscr{F}}\leq
\sup_\epsilon\|g_\epsilon^\alpha-\varphi\|_{\mathscr{F}} +
\|\varphi\|_{\mathscr{F}}\leq
K_1\|\varphi\|_{\mathscr{W}}+\|\varphi\|_{\mathscr{F}}.
\end{displaymath} Now
that $g_\epsilon^\alpha$ satisfies
\begin{displaymath}
\langle-\frac{\partial g_\epsilon^\alpha}{\partial
t},\psi\rangle+\mathcal{A}_\alpha(g_\epsilon^\alpha,\psi)=\frac{1}{\epsilon}\left((g_\epsilon^\alpha-g)^-,\psi\right)_{\mathscr{H}},\quad
\forall \psi\in\mathscr{F},
\end{displaymath}
we have
\begin{equation}\label{supeg}\begin{split}
\|\frac{\partial g_\epsilon^\alpha}{\partial
t}\|_{\mathscr{F}'}&=\sup_{\|\psi\|_{\mathscr{F}}=1}\langle\frac{\partial
g_\epsilon^\alpha}{\partial
t},\psi\rangle\\
&=\sup_{\|\psi\|_{\mathscr{F}}=1}\left(\mathcal{A}_\alpha(g_\epsilon^\alpha,\psi)-\frac{1}{\epsilon}\left((g_\epsilon^\alpha-g)^-,\psi\right)_{\mathscr{H}}\right)\\
&\leq \sup_{\|\psi\|_{\mathscr{F}}=1}
\mathcal{A}_\alpha(g_\epsilon^\alpha,\psi)+\sup_{\|\psi\|_{\mathscr{F}}=1}\frac{1}{\epsilon}\left((g_\epsilon^\alpha-g)^-,\psi\right)_{\mathscr{H}}.
\end{split}
\end{equation}
By the sector condition,
\begin{displaymath}
\mathcal{A}_\alpha(g_\epsilon^\alpha,\psi)\leq
K_2\|g_\epsilon^\alpha\|_{\mathscr{F}}\|\psi\|_{\mathscr{F}},
\end{displaymath} hence
\begin{displaymath}
\sup_{\|\psi\|_{\mathscr{F}}=1}
\mathcal{A}_\alpha(g_\epsilon^\alpha,\psi)\leq
K_2\|g_\epsilon^\alpha\|_{\mathscr{F}}.
\end{displaymath}
On the other hand, by Cauchy-Schwarz inequality, the following
holds:
\begin{displaymath}
\frac{1}{\epsilon}\left((g_\epsilon^\alpha-g)^-,\psi\right)_{\mathscr{H}}\leq\frac{1}{\epsilon}\|g_\epsilon^\alpha-g)^-\|_{\mathscr{H}}\|\psi\|_{\mathscr{H}},
\end{displaymath}
and
\begin{displaymath}
\|\psi\|_{\mathscr{H}}\leq K_3\|\psi\|_{\mathscr{F}},
\end{displaymath}
 hence
\begin{displaymath}
\sup_{\|\psi\|_{\mathscr{F}}=1}\frac{1}{\epsilon}\left((g_\epsilon^\alpha-g)^-,\psi\right)_{\mathscr{H}}\leq
\frac{1}{\epsilon}K_3\|g_\epsilon^\alpha-g)^-\|_{\mathscr{H}}.
\end{displaymath}
Now by taking $\sup_\epsilon$ of (\ref{supeg}) and by
(\ref{assumHnorm}) we get
\begin{displaymath}
\sup_\epsilon \|\frac{\partial g_\epsilon^\alpha}{\partial
t}\|_{\mathscr{F}'} \leq K_2K_1\|\varphi\|_{\mathscr{W}} +
K_2\|\varphi\|_{\mathscr{F}} + KK_3\|g\|_{\mathscr{H}} < \infty.
\end{displaymath}
Therefore
\begin{equation}\label{egwbdd}\begin{split}\sup_\epsilon
\|g_\epsilon^\alpha\|_{\mathscr{W}}&=\sup_\epsilon
\left(\|\frac{\partial g_\epsilon^\alpha}{\partial
t}\|_{\mathscr{F}'}+\|g_\epsilon^\alpha\|_{\mathscr{F}}\right)\\
& \leq K_2K_1\|\varphi\|_{\mathscr{W}} +
K_2\|\varphi\|_{\mathscr{F}} +
KK_3\|g\|_{\mathscr{H}}+K_1\|\varphi\|_{\mathscr{W}}+\|\varphi\|_{\mathscr{F}},
\end{split}
\end{equation} and as a consequence, $e_g\in\mathscr{W}$ by Lemma I.2.12 in \cite{Ma92}.
\end{proof}

\begin{corollary} There exist  constants $K_4,K_5$ such that
$\|e_g\|_{\mathscr{W}}\leq
K_4\|\varphi\|_{\mathscr{W}}+K_5\|g\|_{\mathscr{H}}$.
\end{corollary}
\begin{proof}
This can be seen by Eq.(\ref{egwbdd}) in the proof of Lemma
\ref{egW} and the fact that $\|\varphi\|_{\mathscr{F}}\leq
\|\varphi\|_{\mathscr{W}}$.
\end{proof}

Now we can revise Theorem 1.2 of \cite{Oshima06} and get the
following result.
\begin{corollary}\label{onesided}
Under the assumptions (\ref{gbdd}) and (\ref{assumHnorm}),
$e_g=\lim_{\epsilon\to 0}g_\epsilon^\alpha$ converges increasingly,
strongly in $\mathscr{H}$, and weakly in both $\mathscr{F}$ and
$\mathscr{W}$. Furthermore, $e_g$ is the minimal function of
$\mathscr{P}_\alpha\cap\mathscr{L}_g\cap\mathscr{W}$ satisfying
\begin{equation}\label{viqnormal}
\mathcal{E}_\alpha (e_g,e_g)\leq\mathcal{E}_\alpha(e_g,\psi),\quad
\forall \psi\in\mathscr{L}_g\cap\mathscr{W}.
\end{equation}
\end{corollary}
\begin{proof}
Now that $e_g\in\mathscr{W}$, so $\langle\frac{\partial
e_g}{\partial t},e_g\rangle=0$ (see Lemma 1.1 of \cite{Oshima04}),
hence $\mathcal{A}_\alpha(e_g,e_g)=\mathcal{E}_\alpha(e_g,e_g).$ The
rest of the proof is the same as in \cite{Oshima06}.
\end{proof}

\begin{theorem}\label{optstop} Let $g(z)=g(t,x)$ be a finely continuous function satisfying (\ref{gbdd}).
Assume (\ref{assumHnorm}) and the absolute continuity condition
(\ref{abscont}). Let $e_g$ be the solution of (\ref{viqnormal}), and
$\tilde{e}_g$ be its $\alpha$-excessive regularization. Then
\begin{equation}
\tilde{e}_g(z)=\sup_\sigma J_z(\sigma),\quad \forall
z=(s,x)\in\mathbb{Z},
\end{equation}  where
$J_{z}(\sigma)=J_{(s,x)}(\sigma)=E_{(s,x)}(e^{-\alpha\sigma}g(s+\sigma,X_{s+\sigma}))$.
Furthermore, let the set
$B=\{z\in{\mathbb{Z}}:\tilde{e}_g(z)=g(z)\}$ and let $\sigma_{{B}}$
be the first hitting time of $B$ defined by $\sigma_{{B}}=\inf\{t>
0:\tilde{e}_g(Z_{s+t})=g(Z_{s+t})\}$, then
\begin{equation}\label{optform}
\tilde{e}_g(z)=E_z[e^{-\alpha\sigma_{{B}}} g(Z_{s+\sigma_{{B}}})].
\end{equation}
\end{theorem}
\begin{proof}
Notice that $\varphi\wedge \tilde{e}_g$ is an $\alpha$-potential
dominating $g$, and $\tilde{e}_g$ is the smallest $\alpha$-potential
dominating $g$, we get $\tilde{e}_g\leq \varphi\wedge
\tilde{e}_g\leq \varphi$ $\nu$-a.e., which implies the finiteness of
$\tilde{e}_g$.

Now because $e_g\geq g$ $\nu$ a.e., we have
 $nR_{n+\alpha} e_g(z)\geq
nR_{n+\alpha}g(z)$, $\forall z\in\mathbb{Z},\ n>0$, and this implies
\begin{displaymath}
\tilde{e}_g(z)\geq \lim_{n\to\infty} nR_{n+\alpha}g(z),\quad \forall
z\in\mathbb{Z}.
\end{displaymath} By the absolute continuity condition and applying the dominated convergence theorem, the following
holds,
\begin{displaymath}
\lim_{n\to\infty} nR_{n+\alpha}g(z)=g(z),\quad \forall
z\in\mathbb{Z},
\end{displaymath} therefore $\tilde{e}_g(z)\geq g(z)$,
$\forall z\in\mathbb{Z}$. Then we have
\begin{equation}\label{esup}
\tilde{e}_g(z)\geq
E_z\left(e^{-\alpha\sigma}\tilde{e}_g(Z_{s+\sigma})\right)\geq
E_z\left(e^{-\alpha\sigma}g(Z_{s+\sigma})\right),
\end{equation} for any
stopping time $\sigma$,
 which implies
$\tilde{e}_g(z)\geq J_z(\sigma)$, $\forall z\in\mathbb{Z}$, hence
$\tilde{e}_g(z)\geq \sup_\sigma J_z(\sigma)$, $\forall
z\in\mathbb{Z}$.

%The part of showing $\tilde{e}_g(z)=E_z[e^{-\alpha\sigma_{\dot{B}}}
%g(Z_{s+\sigma_{\dot{B}}})]$ follows the proof of Thm 3.1 in
%\cite{Oshima06}.

Since $e_g$ is a bounded $\alpha$-potential, there exists a positive
Radon measure $\mu^\alpha$ of finite energy such that
\begin{equation}
\mathcal{E}_\alpha(e_g,w)=\int_\mathbb{Z} w(z)\mu^\alpha(dz),\quad
\forall w\in C_0(\mathbb{Z})\cap\mathscr{W},
\end{equation} and $\tilde{e}_g(z)=R_\alpha \mu^\alpha(z)$.

Under the absolute continuity condition (\ref{abscont}) of the
transition function, there exists a positive continuous additive
functional $A_t$ in the strict sense (see Theorem 5.1.6 in
\cite{Fuku11}) such that
\begin{displaymath}
\tilde{e}_g(z)=E_z\left(\int_0^\infty e^{-\alpha t}dA_t\right),\quad
\forall z\in\mathbb{Z}.
\end{displaymath}

Set $B=\{z\in\mathbb{Z}:\tilde{e}_g(z)=g(z)\}$, then
\begin{displaymath}
\int_{B^c}\left(\tilde{e}_g(z)-g(z)\right)\mu^\alpha(dz)=\int_{\mathbb{Z}}\left(\tilde{e}_g(z)-g(z)\right)\mu^\alpha(dz)
=\mathcal{E}_\alpha(e_g,e_g-g).
\end{displaymath}
Since $e_g$ is an $\alpha$-potential, and $e_g-g$ is nonnegative,
$\mathcal{E}_\alpha(e_g,e_g-g)\geq 0$, which implies
$\mathcal{E}_\alpha(e_g,e_g)-\mathcal{E}_\alpha(e_g,g)\geq 0$. On
the other hand, $e_g$ satisfies (\ref{viqnormal}), which implies
$\mathcal{E}_\alpha(e_g,e_g)-\mathcal{E}_\alpha(e_g,g)\leq 0$. Now
it can be concluded that
$\mathcal{E}_\alpha(e_g,e_g)-\mathcal{E}_\alpha(e_g,g)= 0$, hence
$\mu^\alpha(B^c)=0$. Further we get
\begin{displaymath}
E_z\left(\int_0^\infty e^{-\alpha
t}I_{B^c}(Z_{s+t})dA_t\right)=R_\alpha(I_{B^c}\mu)(z)=0,\quad\forall
z\in\mathbb{Z}.
\end{displaymath}
By the strong Markov property, we have for any stopping time
$\sigma\leq\sigma_B$
\begin{equation}\label{sMark}
\tilde{e}_g(z)=E_z\left[\int_0^\sigma e^{-\alpha
t}dA_t\right]+E_z[e^{-\alpha\sigma}\tilde{e}_g(Z_{s+\sigma})],
\end{equation} and because
\begin{displaymath}
0\leq E_z\left[\int_0^\sigma e^{-\alpha t}dA_t\right]\leq
E_z\left(\int_0^\infty e^{-\alpha t}I_{B^c}(Z_{s+t})dA_t\right)=0,
\end{displaymath} we have
$\tilde{e}_g(z)=E_z[e^{-\alpha\sigma}\tilde{e}_g(Z_{s+\sigma})]$,
$\sigma\leq \sigma_B$. By replacing $\sigma$ by $\sigma_B$ and
replacing $\tilde{e}_g(Z_{s+\sigma})$ by $g(Z_{s+\sigma_B})$, we get
$\tilde{e}_g(z)=E_z[e^{-\alpha\sigma_B}g(Z_{s+\sigma_B})]$, and this
together with (\ref{esup}) completes the proof.
\end{proof}

\begin{corollary}
Under the conditions in Theorem \ref{optstop}, $\tilde{e}_g(z)$ is
finely and cofinely continuous for all $z\in\mathbb{Z}$.
\end{corollary}
\begin{proof}
Oshima \cite{Oshima06} has showed that $\tilde{e}_g(z)$ is finely
and cofinely continuous for q.e. z, and under the  conditions in
Theorem \ref{optstop}, we showed that there does not exist the
exceptional set, so $\tilde{e}_g(z)$ is finely and cofinely
continuous for all $z\in\mathbb{Z}$.
\end{proof}
\begin{remark}
Since $X_t$ is a diffusion process, we can see that $\tilde{e}_g(z)$
is continuous along the sample paths, and if $X_t$ is a
non-degenerate Ito diffusion, $\tilde{e}_g(z)$ is continuous. This
gives an alternate proof of the continuity of the value function,
while Palczewski and Stettner \cite{Palczewski10} used a penalty
method to prove it.
\end{remark}
In Palczewski and Stettner's work
\cite{Palczewski10}\cite{Palczewski11}, the optimal policy is to
stop the game at the stopping time $\dot{\sigma}_B=\inf\{t\geq
0:\tilde{e}_g(Z_{s+t})\leq g(Z_{s+t})\}$ or equivalently
$\dot{\sigma}_B=\inf\{t\geq 0:\tilde{e}_g(Z_{s+t})= g(Z_{s+t})\}$.
Notice that $\dot{\sigma}_B\leq\sigma_B$, by Theorem \ref{optstop}
we can see that
\begin{displaymath}
\tilde{e}_g(z)\geq
E_z[e^{-\alpha\dot{\sigma}_B}\tilde{e}_g(Z_{s+\dot{\sigma}_B})]\geq
E_z[e^{-\alpha{\sigma}_B}\tilde{e}_g(Z_{s+{\sigma}_B})]\geq
E_z[e^{-\alpha \sigma_B}g(Z_{s+\sigma_B})]=\tilde{e}_g(z),\ \forall
z,
\end{displaymath} hence
\begin{displaymath}
E_z[e^{-\alpha\dot{\sigma}_B}\tilde{e}_g(Z_{s+\dot{\sigma}_B})]=
E_z[e^{-\alpha{\sigma}_B}\tilde{e}_g(Z_{s+{\sigma}_B})],\ \forall z,
\end{displaymath}
and as a byproduct, we get the following result:
\begin{corollary}
Under the conditions in Theorem \ref{optstop}, there does not exist
the exceptional set of irregular boundary points of $B$.
\end{corollary}
Therefore it is feasible to replace $\sigma_B$ by $\dot{\sigma}_B$
in the results in the rest of this paper.

\begin{remark}
In condition (\ref{assumHnorm}) which is used to characterize the
properties of the value function of optimal stopping,
$g_\epsilon^\alpha$ solves a PDE which involves the generator of the
stochastic process, and the part $(g_\epsilon^\alpha-g)^-$ involves
the function $g$. Since the optimal stopping problem relies on an
underlying process ${\bf M}$ and the reward function $g$, condition
(\ref{assumHnorm}) makes much sense.
\end{remark}

\subsection{The Time Inhomogeneous Zero-sum Game}

In this section we will refine the solution of the two-obstacle
problem (zero-sum game) in\cite{Oshima06}.

Let $g(t,x),h(t,x)\in\mathscr{W}$ be finely continuous functions
satisfying \begin{equation}\label{bddgh}g(t,x)\leq h(t,x), \quad
|g(t,x)|\leq\varphi(t,x), \quad |h(t,x)|\leq \psi(t,x),\ \forall
(t,x)\in\mathbb{Z},\end{equation} where $\varphi,\psi\in\mathscr{W}$
are two bounded $\alpha$-excessive functions. We also assume that
$g,h$ satisfy the condition (\ref{assumHnorm}). Suppose there exist
bounded $\alpha$-excessive functions
$v_1(t,x),v_2(t,x)\in\mathscr{W}$ such that
\begin{equation}\label{sepcond} g(t,x)\leq v_1(t,x)-v_2(t,x)\leq
h(t,x),\ \forall (t,x)\in\mathbb{Z},
\end{equation} in which case we say $g$ and $h$ satisfy the \emph{separability
condition} \cite{Fuku06}.

Define the sequences of $\alpha$-excessive functions $\{\varphi_n\}$
and $\{\psi_n\}$ inductively by
\begin{displaymath}
\varphi_0=\psi_0=0, \psi_n=e_{\varphi_{n-1}-h},
\varphi_n=e_{\psi_{n}+g},\ \ n\geq 1,
\end{displaymath} then the following holds:
\begin{lemma}
Assuming (\ref{sepcond}), then $\varphi_n,\psi_n$ are well defined
and $\lim_{n\to\infty}\varphi_n=\bar{\varphi}$,
$\lim_{n\to\infty}\psi_n=\bar{\psi}$ converge increasingly, strongly
in $\mathscr{H}$ and weakly in both $\mathscr{F}$ and $\mathscr{W}$.
\end{lemma}
\begin{proof} We only need to show the convergence in $\mathscr{W}$
and the rest of this lemma is just Lemma 2.1 in \cite{Oshima06}.
Firstly $\varphi_0=0\leq v_1$ and $\varphi_0\in\mathscr{W}$. Suppose
$\varphi_{n-1}\in\mathscr{W}$ is well defined and satisfies
$\varphi_{n-1}\leq v_1$, then $\varphi_{n-1}-h\leq v_1-h\leq v_2$.
Hence $\psi_n=e_{\varphi_{n-1}-h}\in\mathscr{W}$ is well defined by
Lemma \ref{egW},  and we also have $\psi_n\leq v_2$ since
$e_{\varphi_{n-1}-h}$ is the smallest $\alpha$-potential dominating
$\varphi_{n-1}-h$. Now that $\psi_n+g\leq v_2 + g\leq v_1$, hence
$\varphi_n=e_{\psi_{n}+g}\in \mathscr{W}$ is well defined and is
dominated by $v_1$.

Notice that $\varphi_0\leq\varphi_1$. Suppose $\varphi_{n-1}\leq
\varphi_n$, then $\psi_n=e_{\varphi_{n-1}-h}\leq
e_{\varphi_{n}-h}=\psi_{n+1}$, hence $\varphi_n=e_{\psi_{n}+g}\leq
e_{\psi_{n+1}+g}=\varphi_{n+1}$. Also by Lemma \ref{egW} we get
\begin{displaymath}
\|\varphi_n\|_{\mathscr{W}}=\|e_{\psi_n+g}\|_{\mathscr{W}}\leq
K_4\|v_1\|_{\mathscr{W}}+K_5\|\psi_n+g\|_{\mathscr{H}}.
\end{displaymath} Notice that $g\leq \psi_n+g\leq v_1$, hence
$\|\psi_n+g\|_{\mathscr{H}}$ is uniformly bounded in $n$, and as a
consequence, $\|\varphi_n\|_{\mathscr{W}}$ is uniformly bounded in
$n$. In a similar manner we can show that $\|\psi_n\|_{\mathscr{W}}$
is uniformly bounded. The convergence of $\varphi_n,\psi_n$ in
$\mathscr{W}$ follows by Lemma I.2.12 in \cite{Ma92}.
\end{proof}

\begin{corollary}
Under the separability condition, $\bar{\varphi}=e_{\bar{\psi}+g}$,
$\bar{\psi}=e_{\bar{\varphi}-h}$, and they satisfy
\begin{equation}\label{2ineq}\begin{split}
\mathcal{E}_\alpha(\bar{\varphi},\bar{\varphi})&\leq\mathcal{E}_\alpha(\bar{\varphi},w),\quad\forall
w\in\mathscr{L}_{\bar{\psi}+g}\cap\mathscr{W},\\
\mathcal{E}_\alpha(\bar{\psi},\bar{\psi})&\leq\mathcal{E}_\alpha(\bar{\psi},w),\quad\forall
w\in\mathscr{L}_{\bar{\varphi}-h}\cap\mathscr{W}.
\end{split}
\end{equation}
\end{corollary}
\begin{proof}
Since $\bar{\varphi}$ is an $\alpha$-potential dominating
$\bar{\psi}+g$, we get $e_{\bar{\psi}+g}\leq\bar{\varphi}$. On the
other hand, $\bar{\varphi}=\lim_{n\to\infty}
\varphi_n=\lim_{n\to\infty} e_{\psi_n+g}\leq e_{\bar{\psi}+g}$,
hence $\bar{\varphi}=e_{\bar{\psi}+g}$. Similarly
$\bar{\psi}=e_{\bar{\varphi}-h}$. The proof of (\ref{2ineq}) is
immediate by Corollary \ref{onesided}.
\end{proof}

\begin{corollary}\label{uniq}
If a pair of $\alpha$-excessive functions $(V_1,V_2)$ satisfy $g\leq
V_1-V_2\leq h$, then $\bar{\varphi}\leq V_1$, $\bar{\psi}\leq V_2$,
and $\bar{w}:=\bar{\varphi}-\bar{\psi}$ is the unique function in
$\mathscr{J}$ satisfying
\begin{equation}\label{vineq}
\mathcal{E}_\alpha(\bar{w},\bar{w})\leq
\mathcal{E}_\alpha(\bar{w},w),\quad\forall w\in\mathscr{J},\ g\leq
w\leq h,
\end{equation} where $\mathscr{J}=\{w=\varphi_1-\varphi_2+v:\varphi_1,\varphi_2\in\mathscr{W} \text{ are }\alpha\text{-potentials},
v\in\mathscr{W}\}$.
\end{corollary}
\begin{proof}
Clearly $\varphi_{n-1}-h\leq\psi_n$ and
$\psi_n+g\leq\bar{\varphi}_n$, hence
$g\leq\bar{\varphi}-\bar{\psi}\leq h$. If $g,h$ satisfy the
separability condition with respect to $V_1,V_2$, then we would have
$\varphi_n\leq V_1$ and $\psi_n\leq V_2$, and as a consequence
$\bar{\varphi}\leq V_1$, $\bar{\psi}\leq V_2$.

Now (\ref{vineq}) is equivalent to
\begin{displaymath}\mathcal{E}_\alpha(\bar{\varphi},\bar{\varphi})+\mathcal{E}_\alpha(\bar{\psi},\bar{\psi})\leq \mathcal{E}_\alpha(\bar{\varphi},w+\bar{\psi})+\mathcal{E}_\alpha(\bar{\psi},\bar{\varphi}-w),\ g\leq w\leq h,\end{displaymath}
which holds by (\ref{2ineq}). Suppose there are two solutions
$\bar{w}_1,\bar{w}_2\in\mathscr{J}$ satisfying (\ref{vineq}). Notice
that $\bar{w}_1-\bar{w}_2\in\mathscr{W}$ and $\langle\frac{\partial
(\bar{w}_1-\bar{w}_2)}{\partial t},\bar{w}_1-\bar{w}_2\rangle=0$, so
\begin{displaymath}
\langle\frac{\partial \bar{w}_1}{\partial t},\bar{w}_2\rangle+
\langle\frac{\partial \bar{w}_2}{\partial t},\bar{w}_1\rangle=0,
\end{displaymath} and consequently
\begin{displaymath}
\mathcal{A}_\alpha(\bar{w}_1,\bar{w}_2)+\mathcal{A}_\alpha(\bar{w}_2,\bar{w}_1)=\mathcal{E}_\alpha(\bar{w}_1,\bar{w}_2)+\mathcal{E}_\alpha(\bar{w}_2,\bar{w}_1).
\end{displaymath} Therefore,
\begin{displaymath}\begin{split}
\mathcal{A}_\alpha(\bar{w}_1-\bar{w}_2,\bar{w}_1-\bar{w}_2)&=\mathcal{A}_\alpha(\bar{w}_1,\bar{w}_1)+\mathcal{A}_\alpha(\bar{w}_2,\bar{w}_2)-\mathcal{A}_\alpha(\bar{w}_1,\bar{w}_2)-\mathcal{A}_\alpha(\bar{w}_2,\bar{w}_1)\\
&=\mathcal{A}_\alpha(\bar{w}_1,\bar{w}_1)+\mathcal{A}_\alpha(\bar{w}_2,\bar{w}_2)-\mathcal{E}_\alpha(\bar{w}_1,\bar{w}_2)-\mathcal{E}_\alpha(\bar{w}_2,\bar{w}_1)\leq
0,
\end{split}
\end{displaymath} which implies that $\bar{w}_1=\bar{w}_2$ a.e.
\end{proof}

Let $\tilde{\bar{\varphi}},\tilde{\bar{\psi}},\tilde{\bar{w}}$ be
the $\alpha$-excessive modifications of
$\bar{\varphi},\bar{\psi},\bar{w}$, respectively. We further define
for arbitrary pair of stopping times $\tau,\sigma$ the payoff
function $J_z(\tau,\sigma)$ as
\begin{equation}\label{payoffh0}
J_z(\tau,\sigma)=E_z\left[e^{-\alpha(\tau\wedge\sigma)}\left(g(Z_{s+\sigma})I_{\tau>\sigma}+h(Z_{s+\tau})I_{\tau\leq\sigma}\right)\right],\quad
z\in\mathbb{Z}.
\end{equation} Then we have the following result:

\begin{theorem}\label{saddle0} Assuming the separability condition on
$g,h$ and conditions (\ref{abscont}) (\ref{assumHnorm}). There
exists a finite finely and cofinely continuous function
$\tilde{\bar{w}}(z)\in\mathscr{J}$ satisfying (\ref{vineq}) and the
identity
\begin{equation}\label{vsaddle}
\tilde{\bar{w}}(z)=\sup_\sigma\inf_\tau
J_z(\tau,\sigma)=\inf_\tau\sup_\sigma J_z(\tau,\sigma),\quad \forall
z=(s,x)\in\mathbb{Z},
\end{equation}where $\sigma,\tau$ range over all stopping times.
Moreover, the pair $\hat{\tau},\hat{\sigma}$ defined by
\begin{displaymath}
\hat{\tau}=\inf\{t>
0:\bar{w}(Z_{s+t})=h(Z_{s+t})\},\quad\hat{\sigma}=\inf\{t>
0:\bar{w}(Z_{s+t})=g(Z_{s+t})\}
\end{displaymath} is the saddle point of the game in the sense that
\begin{displaymath}
J_z(\hat{\tau},\sigma)\leq J_z(\hat{\tau},\hat{\sigma})\leq
J_z(\tau,\hat{\sigma}),\quad z\in\mathbb{Z},
\end{displaymath} for all stopping times $\tau,\sigma$.
\end{theorem}
\begin{proof} We only need to prove (\ref{vsaddle}).
By Theorem \ref{optstop}, for any $z\in\mathbb{Z}$ we have
\begin{equation}\label{ubarvbar}\begin{split}
\tilde{\bar{\varphi}}(z)&=\sup_\sigma
E_z[e^{-\alpha\sigma}(\tilde{\bar{\psi}}+g)(Z_{s+\sigma})]=E_z[e^{-\alpha\hat{\sigma}}(\tilde{\bar{\psi}}+g)(Z_{s+\hat{\sigma}})],\\
\tilde{\bar{\psi}}(z)&=\sup_\tau
E_z[e^{-\alpha\tau}(\tilde{\bar{\varphi}}-H)(Z_{s+\tau})]=E_z[e^{-\alpha\hat{\tau}}(\tilde{\bar{\varphi}}-h)(Z_{s+\hat{\tau}})],
\end{split}
\end{equation} and for any stopping times
$\sigma\leq\hat{\sigma}$, $\tau\leq\hat{\sigma}$,
\begin{displaymath}
\tilde{\bar{\varphi}}(z)=E_z[e^{-\alpha\sigma}\tilde{\bar{\varphi}}(Z_{s+\sigma})],\
\tilde{\bar{\psi}}(z)=E_z[e^{-\alpha\tau}\tilde{\bar{\psi}}(Z_{s+\tau})],\quad\forall
z=(s,x)\in\mathbb{Z}.
\end{displaymath} From (\ref{sMark}), we could take $\{e^{-\alpha
t}\bar{\varphi}(Z_{s+t})\}$ and $\{e^{-\alpha t}\bar{v}(Z_{s+t})\}$
as non-negative $P_z$-supermartingales, therefore, for any
$z\in\mathbb{Z}$ and any stopping times $\tau,\sigma$, we have
\begin{displaymath}
\tilde{\bar{\varphi}}(z)\geq
E_z[e^{-\alpha\sigma}\tilde{\bar{\varphi}}(Z_{s+\sigma})],\
\tilde{\bar{\psi}}(z)\geq
E_z[e^{-\alpha\tau}\tilde{\bar{\psi}}(Z_{s+\tau})].
\end{displaymath} Consequently, for any $z\in\mathbb{Z}$,
\begin{displaymath}\begin{split}
\tilde{\bar{w}}(z)&=\tilde{\bar{\varphi}}(z)-\tilde{\bar{\psi}}(z)\leq
E_z[e^{-\alpha(\hat{\sigma}\wedge\tau)}\tilde{\bar{\varphi}}(Z_{s+\hat{\sigma}\wedge\tau})]-E_z[e^{-\alpha(\hat{\sigma}\wedge\tau)}\tilde{\bar{\psi}}(Z_{s+\hat{\sigma}\wedge\tau})]\\
&=E_z[e^{-\alpha(\hat{\sigma}\wedge\tau)}\tilde{\bar{w}}(Z_{s+\hat{\sigma}\wedge\tau})]\leq
E_z\left[e^{-\alpha(\tau\wedge\hat{\sigma})}\left(g(Z_{s+\sigma})I_{\tau>\hat{\sigma}}+h(Z_{s+\tau})I_{\tau\leq\hat{\sigma}}\right)\right]=J_z(\tau,\hat{\sigma}),
\end{split}
\end{displaymath} where the last inequality is due to the fact that
$g(z)\leq\tilde{\bar{w}}(z)\leq h(z)$, $\forall z\in\mathbb{Z}$ and
(\ref{ubarvbar}). In a similar manner, we can prove that
$\tilde{\bar{w}}\geq J_z(\hat{\tau},\sigma)$, and this completes the
proof.
\end{proof}

\subsection{Time Inhomogeneous Stopping Game with Holding
Cost}\label{tistophold} Usually the optimal stopping games involve a
holding cost function $f\in\mathscr{H}$, see, e.g.,
\cite{Palczewski10}, and the return functions become
\begin{equation}\label{costh1}
J_{(s,{ x})}^f(\sigma)=E_{(s,{ x})}\left(\int_0^{\sigma}e^{-\alpha
t}f(s+t,{
X}_{s+t})dt+e^{-\alpha\sigma}g(s+\sigma,X_{s+\sigma})\right) ,
\end{equation} and
\begin{equation}\label{costh2}\begin{split}
J_{(s,{ x})}^f(\sigma,\tau)&=E_{(s,{
x})}\left(\int_0^{\sigma\wedge\tau}e^{-\alpha t}f(s+t,{
X}_{s+t})dt\right)\\
&\ \ +E_{(s,{
x})}\left(e^{-\alpha(\sigma\wedge\tau)}\left(g(s+\sigma,{
X}_{s+\sigma})I_{\sigma<\tau}+h(s+\tau,{
X}_{s+\tau})I_{\tau\leq\sigma}\right)\right),
\end{split}
\end{equation}
but this model can be essentially reduced to the classical stopping
problem by taking $\hat{g}=g-R_\alpha f$ and $\hat{h}=h-R_\alpha f$
instead of $g$ and $h$ respectively, where $R_\alpha$ is the
resolvent and $R_\alpha f$ is considered as a version of $G_\alpha
f\in\mathscr{W}$. We assume that conditions (\ref{gbdd})
(\ref{assumHnorm})
 also apply to $\hat{g}$ for the optimal stopping game (and similarly conditions
(\ref{bddgh})(\ref{sepcond}) apply to $\hat{g},\hat{h}$ for the
zero-sum game).

\begin{theorem}\label{optstophold} Let $g$ be a finely continuous function satisfying (\ref{gbdd}).
Assume (\ref{assumHnorm}) on ${g}$ and the absolute continuity
condition (\ref{abscont}) on $p_t$. Let $e_{{g}}^f$ be the solution
of \begin{equation}\label{holdineq}
\mathcal{E}_\alpha(e_{{g}}^f,\psi-e_{{g}}^f)\geq
(f,\psi-e_{{g}}^f)_{\mathscr{H}},\quad \forall \psi\in
\mathscr{L}_{{g}}\cap\mathscr{W},
\end{equation} and let $\tilde{e}_{{g}}^f$ be its $\alpha$-excessive
regularization. Then
\begin{equation}
\tilde{e}_{{g}}^f(z)=\sup_\sigma J_z^f(\sigma),\quad \forall
z=(s,x)\in\mathbb{Z},
\end{equation} where $J_{z}^f(\sigma)$ is defined as in (\ref{costh1}), and $\tilde{e}_{{g}}^f(z)$ is finely and cofinely continuous. Furthermore, let the set
$B=\{z\in{\mathbb{Z}}:\tilde{e}_{{g}}^f(z)={g}(z)\}$ and let
$\sigma_{{B}}$ be the first hitting time of $B$ defined by
$\sigma_{{B}}=\inf\{t> 0:\tilde{e}_{{g}}^f(Z_{s+t})={g}(Z_{s+t})\}$,
then
\begin{equation}\label{holdoptform}
\tilde{e}_{{g}}^f(z)=E_z[e^{-\alpha\sigma_{{B}}}
{g}(Z_{s+\sigma_{{B}}})].
\end{equation}
\end{theorem}
\begin{proof} Define the function
\begin{displaymath}
J_z^{f0}(\sigma)=E_z(e^{-\alpha\sigma}
\hat{g}(s+\sigma,X_{s+\sigma})),
\end{displaymath} where $\hat{g}=g-R_\alpha f$, and let $\tilde{e}_{\hat{g}}^f=\sup_\sigma
J_z^{f0}(\sigma)$, then by Theorem \ref{optstop}, $e_{\hat{g}}^f$
solves
\begin{equation}\label{tempvieq}
\mathcal{E}_\alpha(e_{\hat{g}}^f,\hat{\psi}-e_{\hat{g}}^f)\geq
0,\quad \forall \hat{\psi}\in\mathscr{L}_{\hat{g}}\cap\mathscr{W},
\end{equation} and the optimal stopping time is defined by $\sigma_{{B}}=\inf\{t>
0:\tilde{e}_{\hat{g}}^f(Z_{s+t})=\hat{g}(Z_{s+t})\}$.

By Dynkin's formula,
\begin{displaymath}
E_{(s,x)}\left(\int_0^\sigma e^{-\alpha
t}f(s+t,X_{s+t})dt\right)=R_\alpha
f(s,x)-E_{(s,x)}\left(e^{-\alpha\sigma}R_\alpha
f(s+\sigma,X_{s+\sigma})\right),
\end{displaymath} which leads to
\begin{displaymath}
J_z^f(\sigma)=J_z^{f0}(\sigma)+R_\alpha f(z),
\end{displaymath} hence $e_g^f(z)=e_{\hat{g}}^f(z)+R_\alpha
f(z)$.

Now let $e_{\hat{g}}^f(z)=e_g^f(z)-R_\alpha f(z)$,
$\hat{\psi}=\psi-R_\alpha f$ in (\ref{tempvieq}) we get
\begin{equation}\label{transineq}
\mathcal{E}_\alpha(e_{{g}}^f-G_\alpha f,\psi-e_{{g}}^f)\geq 0.
\end{equation} Since
$\mathcal{E}_\alpha(G_\alpha
f,\psi-e_{\hat{g}}^f)=(f,\psi-e_{\hat{g}}^f)_{\mathscr{H}}$, this
proves (\ref{holdineq}). Also notice that the optimal stopping time
can be written as $\sigma_{{B}}=\inf\{t>
0:\tilde{e}_{{g}}^f(Z_{s+t})={g}(Z_{s+t})\}$, and this completes the
proof.
\end{proof}

Similarly we can modify Theorem \ref{saddle0} and get the following
result:
\begin{theorem}\label{zerosumhold}
Let $g,h$ be finely continuous functions satisfying (\ref{bddgh})
and (\ref{sepcond}). Assume (\ref{assumHnorm}) on $g,h$ and the
absolute continuity condition (\ref{abscont}) on $p_t$. Then there
exists a finite finely and cofinely continuous function
$\tilde{\bar{w}}^f\in\mathscr{J}$, ${g}(z)\leq
\tilde{\bar{w}}^f(z)\leq {h}(z)$, such that
\begin{equation}\label{vieqzerosumhold}
\mathcal{E}_\alpha({\bar{w}}^f,w-{\bar{w}}^f)\geq
(f,w-{\bar{w}}^f)_{\mathscr{H}}, \quad \forall w\in\mathscr{J},\
{g}\leq w\leq {h},
\end{equation}
 and
\begin{equation}\label{vsaddlehold}
\tilde{\bar{w}}^f(z)=\sup_\sigma\inf_\tau
J_z^f(\tau,\sigma)=\inf_\tau\sup_\sigma J_z^f(\tau,\sigma),\quad
\forall z=(s,x)\in\mathbb{Z},
\end{equation}where $J_z^f(\tau,\sigma)$ was given by (\ref{costh2}) and $\sigma,\tau$ range over all stopping times.
Moreover, the pair $\hat{\tau},\hat{\sigma}$ defined by
\begin{displaymath}
\hat{\tau}=\inf\{t>
0:\bar{w}^f(Z_{s+t})={h}(Z_{s+t})\},\quad\hat{\sigma}=\inf\{t>
0:\bar{w}^f(Z_{s+t})={g}(Z_{s+t})\}
\end{displaymath} is the saddle point of the game in the sense that
\begin{displaymath}
J_z^f(\hat{\tau},\sigma)\leq J_z^f(\hat{\tau},\hat{\sigma})\leq
J_z^f(\tau,\hat{\sigma}),\quad z\in\mathbb{Z},
\end{displaymath} for all stopping times $\tau,\sigma$.
\end{theorem}

As an extension of Corollary \ref{uniq}, we have the following:
\begin{corollary}
The variational inequality (\ref{vieqzerosumhold}) has a unique
solution.
\end{corollary}
\begin{proof}
The case where $f=0$ was proved in Corollary \ref{uniq}. For a
general $f\in\mathscr{H}$, notice again that
$(f,w-\bar{w}^f)_{\mathscr{H}}=\mathcal{E}_\alpha(G_\alpha
f,w-\bar{w}^f)$, we get
\begin{displaymath}
\mathcal{E}_\alpha(\bar{w}^f-G_\alpha f,(w-G_\alpha
f)-(\bar{w}^f-G_\alpha f))\geq 0,\quad \forall w\in\mathscr{J},\
g\leq w\leq h.
\end{displaymath} Let $\hat{\bar{w}}^f=\bar{w}^f-G_\alpha f$, $\hat{w}=w-G_\alpha
f$, $\hat{g}=g-G_\alpha f$, $\hat{h}=h-G_\alpha f$, we get
\begin{displaymath}
\mathcal{E}_\alpha(\hat{\bar{w}}^f,\hat{w}-\hat{\bar{w}}^f)\geq
0,\quad \forall \hat{w}\in\mathscr{J},\
\hat{g}\leq\hat{w}\leq\hat{h},
\end{displaymath} which has a unique solution in view of Corollary
\ref{uniq}.
\end{proof}
%\begin{proof}
%Since $\bar{w}^f$ and $w$ belong to $\mathscr{W}$,
%$\mathcal{E}_\alpha(\bar{w}^f,w-\bar{w}^f)$ reduces to
%$\mathcal{A}_\alpha(\bar{w}^f,w-\bar{w}^f)$, and the rest of the
%proof is the same as in Proposition 2.1 of \cite{Fuku02}.
%\end{proof}

\section{Time Inhomogeneous Stopping Games of Ito Diffusion}\label{exam}

In this section we are concerned with a multi-dimensional time
inhomogeneous Ito diffusion:
\begin{equation}\label{omodel}
d X_t= b(t, X_t)dt+a(t, X_t)d B_t,\ X_s = x,
\end{equation} where
\begin{displaymath}
{ X}_t=\left(\begin{array}{c}X_{1t}\\
\vdots\\
X_{nt}
\end{array}\right), b=\left(\begin{array}{c}b_1\\
\vdots\\
b_n
\end{array}\right), a=\left(\begin{array}{ccc} a_{11} &\  \cdots & a_{1m}\\
\vdots &  & \vdots\ \\
a_{n1} &\ \cdots & a_{nm}\end{array}\right), {
B}_t=\left(\begin{array}{c}B_{1t}\\
\vdots\\
B_{mt}\end{array}\right), m\geq n,
\end{displaymath} and $a_{i,j},b_i,i=1,2,...,n, j = 1,2,...m,$  are continuous functions of
$t$ and $X_t$. Define the square matrix $[A_{i,j}]={\bf
A}=\frac{1}{2}aa^T$. We assume ${\bf A}$ is uniformly
non-degenerate, and $a,b$ satisfy the usual Lipschitz conditions so
that (\ref{omodel}) has a unique strong solution. $B_t$ in
(\ref{omodel}) is assumed to be the standard multi-dimensional
Brownian motion. Thus we are given a system $(\Omega,
\mathcal{F},\mathcal{F}_t, { X},\theta_t,P_{ x})$, where
$(\Omega,\mathcal{F})$ is a measurable space, ${ X}={ X}(\omega)$ is
a mapping of $\Omega$ into $C(\mathbb{R}^n)$, $\mathcal{F}_t$ is the
sigma algebra generated by $ X_s(s\leq t)$, and $\theta_t$ is a
shift operator in $\Omega$ such that $ X_s(\theta_t\omega)=
X_{s+t}(\omega)$. Here $P_{ x}$(${ x}\in\mathbb{R}$) is a family of
measures under which $\{{ X}_t,t\geq 0\}$ is a diffusion with
initial state ${ x}$.

At each time $t$, define the infinitesimal generator ${L}^{(t)}$ as
\begin{equation}\label{infgen}
{L}^{(t)}u(x)=\sum_{i=1}^n b_i(t,x)\frac{\partial u}{\partial
x_i}+\sum_{i,j} A_{i,j}(t,x)\frac{\partial^2 u}{\partial x_i
\partial x_j}.
\end{equation}  Let the positive Radon measure
${\bf m}(dx)=\rho^{(t)}(x)dx$, where $\rho^{(t)}$ satisfies
\begin{equation}
{\bf A}\nabla\rho^{(t)}=\rho^{(t)}\mu,\quad \forall t,
\end{equation} and $\mu_i=b_i-\sum_{j=1}^n\frac{\partial A_{ji}}{\partial
x_j},i=1,2,...,n.$ Notice that when $a$ and $b$ in (\ref{omodel})
are constants, $\rho^{(t)}$ reduces to
\begin{displaymath}
\rho^{(t)}(x)=e^{({\bf A}^{-1}b)\cdot x}.
\end{displaymath} Thus the associated Dirichlet form $({E}^{(t)},F)$ densely embedded in
$H=L^2(\mathbb{R}^n;{\bf m})$ is then given by
\begin{equation}\label{Dformt}
{E}^{(t)}(u,v)=\int_{\mathbb{R}^n} \nabla u(x)\cdot {\bf A} \nabla
v(x) {\bf m}(d{ x}), \quad u,v\in F,
\end{equation}
where
\begin{displaymath}
F=\{u\in H:\ u\text{ is continuous},\
\|u\|_F^2=E_1^{(0)}(u,u)<\infty\}.
\end{displaymath}
Now we can define the sets $\mathscr{F},\mathscr{H},\mathscr{W}$ in
the same way as in Section \ref{tidirichlet}, and define the time
inhomogeneous Dirichlet form $\mathcal{E}_\alpha$ as well.

Since $X_t$ is a non-degenerate Ito diffusion, the absolute
continuity condition on its transition function automatically holds,
and for the same reason, the fine and cofine continuity notion can
be changed to the usual continuity.

Let $f\in\mathscr{H}, g\in\mathscr{W}$ be continuous functions
satisfying the conditions as in Section \ref{tistophold}, and define
the return function $J_z^f(\sigma)$ as in (\ref{costh1}), then we
have the following result:

\begin{theorem}\label{optstopholdIto} Assume (\ref{gbdd}) (\ref{assumHnorm})  on ${g}$ and the absolute
continuity condition (\ref{abscont}) on $p_t$. Let $e_{{g}}^f$ be
the solution of \begin{equation}\label{holdineqIto}
\mathcal{E}_\alpha(e_{{g}}^f,\psi-e_{{g}}^f)\geq
(f,\psi-e_{{g}}^f)_{\mathscr{H}},\quad \forall \psi\in
\mathscr{L}_{{g}}\cap\mathscr{W},
\end{equation} and let $\tilde{e}_{{g}}^f$ be its $\alpha$-excessive
regularization. Then
\begin{equation}
\tilde{e}_{{g}}^f(z)=\sup_\sigma J_z^f(\sigma),\quad \forall
z=(s,x)\in\mathbb{Z},
\end{equation} where $J_{z}^f(\sigma)$ is defined as (\ref{costh1}), and $\tilde{e}_{{g}}^f(z)$ is continuous. Furthermore, let the set
$B=\{z\in{\mathbb{Z}}:\tilde{e}_{{g}}^f(z)={g}(z)\}$ and let
$\sigma_{{B}}$ be the first hitting time of $B$ defined by
$\sigma_{{B}}=\inf\{t> 0:\tilde{e}_{{g}}^f(Z_{s+t})={g}(Z_{s+t})\}$,
then
\begin{equation}\label{optformIto}
\tilde{e}_{{g}}^f(z)=E_z[e^{-\alpha\sigma_{{B}}}
{g}(Z_{s+\sigma_{{B}}})].
\end{equation}
\end{theorem}

For the zero-sum game of Ito diffusion with the return function
$J_{z}^f(\sigma,\tau)$ as defined in (\ref{costh2}), we have the
following result:
\begin{theorem}\label{zerosumholdIto}
Let $g,h$ be continuous functions satisfying (\ref{bddgh}) and
(\ref{sepcond}). Assume (\ref{assumHnorm}) on $g,h$ and the absolute
continuity condition (\ref{abscont}) on $p_t$. Then there exists a
finite and continuous function $\tilde{\bar{w}}^f\in\mathscr{J}$,
${g}(z)\leq \tilde{\bar{w}}^f(z)\leq {h}(z)$, such that
\begin{equation}\label{vieqzerosumholdIto}
\mathcal{E}_\alpha({\bar{w}}^f,w-{\bar{w}}^f)\geq
(f,w-{\bar{w}}^f)_{\mathscr{H}}, \quad \forall w\in\mathscr{J},\
{g}\leq w\leq {h},
\end{equation}
 and
\begin{equation}\label{vsaddleholdIto}
\tilde{\bar{w}}^f(z)=\sup_\sigma\inf_\tau
J_z^f(\tau,\sigma)=\inf_\tau\sup_\sigma J_z^f(\tau,\sigma),\quad
\forall z=(s,x)\in\mathbb{Z},
\end{equation}where $\sigma,\tau$ range over all stopping times and $J_{z}^f(\sigma,\tau)$ is defined in (\ref{costh2}).
Moreover, the pair $\hat{\tau},\hat{\sigma}$ defined by
\begin{displaymath}
\hat{\tau}=\inf\{t>
0:\bar{w}^f(Z_{s+t})={h}(Z_{s+t})\},\quad\hat{\sigma}=\inf\{t>
0:\bar{w}^f(Z_{s+t})={g}(Z_{s+t})\}
\end{displaymath} is the saddle point of the game in the sense that
\begin{displaymath}
J_z^f(\hat{\tau},\sigma)\leq J_z^f(\hat{\tau},\hat{\sigma})\leq
J_z^f(\tau,\hat{\sigma}),\quad z\in\mathbb{Z},
\end{displaymath} for all stopping times $\tau,\sigma$.
\end{theorem}

\noindent{\bf Acknowledgement} The author is very grateful to
Dr.Yoichi Oshima, Dr.Wilhelm Stannat and Dr.Zhen-Qing Chen for
valuable discussions on Dirichlet form. All the possible errors are
certainly mine.

\end{document}